\documentclass{birkjour}
\usepackage{graphicx}
\usepackage{amsmath}
\usepackage{amssymb}
\usepackage{amsfonts}
\usepackage{dbl12}
\usepackage{amscd}
\usepackage{epsfig}
\usepackage{verbatim}
\usepackage{algorithm}
\usepackage{algorithmic}
\usepackage[all]{xy}

\newtheorem{thm}{Theorem}
\newtheorem{prop}{Proposition}
\newtheorem{lemma}{Lemma}

\theoremstyle{definition}
\newtheorem{defn}{Definition}

\newcommand{\ignore}[1]{}
\newcommand{\MM}{\mathcal{M}}
\newcommand{\F}{\mathcal{F}}
\renewcommand{\H}{\mathcal{H}}

\newcommand{\G}{\mathcal{G}}
\renewcommand{\S}{\mathcal{S}}
\newcommand{\C}{\mathcal{C}}
\newcommand{\T}{\mathcal{T}}
\newcommand{\E}{\mathcal{E}}
\newcommand{\D}{D}
\renewcommand{\P}{\mathcal{P}}
\newcommand{\R}{R}

\begin{document}

\title{Affine and projective tree metric theorems}
\author{Aaron Kleinman}
\address{Department of Mathematics, UC Berkeley}
\email{\{kleinman\}@math.berkeley.edu}
\author{Matan Harel}
\address{Department of Mathematics, NYU}
\author{Lior Pachter}
\address{Departments of Mathematics and Molecular \& Cell Biology, UC Berkeley}
\email{\{lpachter\}@math.berkeley.edu}

\begin{abstract}
The tree metric theorem provides a combinatorial four point condition
that characterizes dissimilarity maps derived from pairwise compatible
split systems. A related weaker four point condition
characterizes dissimilarity maps derived from circular split systems
known as Kalmanson metrics. The tree metric theorem was first discovered in
the context of phylogenetics and forms the basis of many tree
reconstruction algorithms, whereas Kalmanson metrics were first
considered by computer scientists,
and are notable in that they are a non-trivial class of metrics for which the traveling salesman problem is tractable. 

We present a unifying framework for these theorems based on
combinatorial structures that are used for graph planarity
testing. These are (projective) PC-trees, and
their affine analogs, PQ-trees. In the projective case, we generalize
a number of concepts from clustering theory, including hierarchies,
pyramids, ultrametrics and Robinsonian matrices, and the theorems that
relate them. As with tree metrics and ultrametrics, the link between
PC-trees and PQ-trees is established via the Gromov product.
\end{abstract}
\thanks{We thank Laxmi Parida for introducing us to the applications of PQ
trees in biology during a visit to UC Berkeley in 2008. Thanks also to
an anonymous reviewer whose careful reading of an initial draft of the
paper helped us greatly. AK was funded by an NSF graduate research fellowship.}
\maketitle
\section{Introduction}
In his ``Notebook B: Transmutation of Species'' (1837), Charles Darwin
drew a single figure to illustrate the
shared ancestry of extant species (Figure 1). That figure is a pictorial depiction
of a graph known as a {\em rooted $X$-tree}.
\begin{defn}\label{def:rootedXTree}
An {\em $X$-tree} $\T = (T,\phi)$ is a pair where $T = (V,E)$ is a tree and $\phi: X \to V$ is a 
bijection from $X$ to the leaves of $T$. Two $X$-trees $\T_1 = (T_1, \phi_1)$,  $\T_2 = (T_2, \phi_2)$
are isomorphic if there exists a graph isomorphism $\Phi: T_1 \to T_2$ such that $\phi_2 = \Phi \circ \phi_1$.
A {\em rooted $X$-tree} is an $X$-tree with a distinguished vertex $r$.
\end{defn}
\begin{figure}
\label{fig1}
\begin{center}
\includegraphics[scale = 0.8]{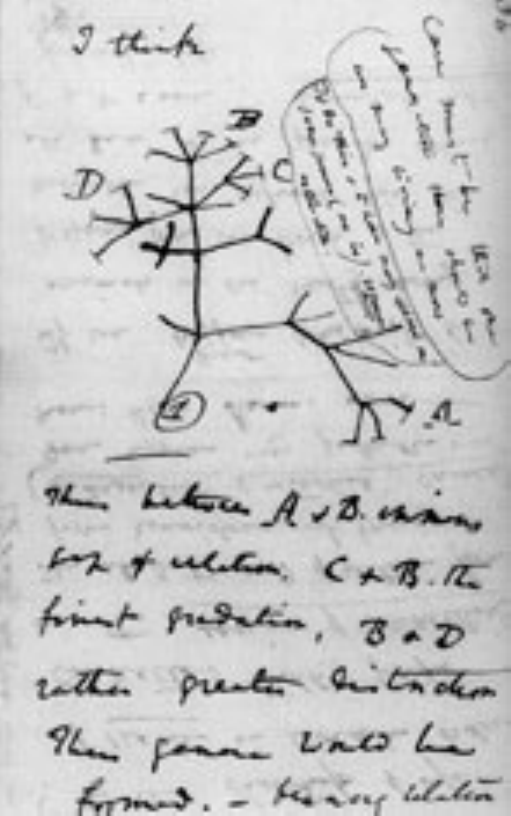}
\end{center}
\caption{A page from Darwin's Notebook B on the Transmutation of Species
  showing a rooted $X$-tree (the root is labeled with ``1'')
  where the set $X=\{A,B,C,D\}$.}
\end{figure}
The biological interpretation of rooted $X$-trees lies in the
identification of the leaves with extant species, and the vertices along a path from a 
leaf to the root as ancestral species. Although the validity of trees in describing the
ancestry of species has been debated \cite{Huson2005b}, trees are now used to
describe the shared ancestry of individual nucleotides in genomes and
they are perfectly suited for that purpose \cite{Dewey2006}.

In phylogenetics, it is desirable to associate
lengths with the edges of trees. Such lengths may correspond to time (in
years), or to the number of mutations (usually an estimate based on a
statistical model). This leads to the notion of a tree metric, that is
conveniently understood via the notion of a weighted split system.
Here and in what follows, for simplicity we often take $X = \{1,2,\ldots,n\}$.
\begin{defn}\label{def:split}
A \emph{split} $S$ of $X$ is a partition $A|B$ of $X$ into two nonempty subsets. The corresponding \emph{split pseudometric} is given by 
\[
\D_S(i,j) = \begin{cases} 1 & \mathrm{if\ } |\{i,j\} \cap A|=1,
 \\ 0 & \mathrm{otherwise}. \end{cases}
\]
A split $A|B$ is \emph{trivial} if $|A|=1$ or $|B|=1$. A \emph{split system} is a set of splits containing all the trivial splits.
\end{defn}
Removing an edge of an $X$-tree  disconnects the tree into two pieces and thus gives a split $S_e$ of $X$. We can associate a split system 
$\{S_e | e \in E(\T)\}$  to $\T$ by doing this for each edge.
\begin{defn}\label{def:dissmap}\label{def:Sadditive}
A \emph{dissimilarity map} is a function $\D:X \times X \to \mathbf{R_{\geq 0}}$ such that $\D(i,i)=0$ and $\D(i,j)=\D(j,i)$ $\forall i,j \in X$.
A dissimilarity map $\D$ is \emph{$\S$-additive} if there is a split system $\S$ such that
\[
\D = \sum_{S \in \S} w(S) \D_S
\]
for some non-negative weighting function $w:\S \to \mathbf{R}_{\geq 0}$. If $\S$ is the split system associated to an $X$-tree $\T$, we say $\D$ is $\T$-additive. $\D$ is a \emph{tree metric} if it is $\T$-additive for some $X$-tree.
\end{defn}
The following classic theorem precisely characterizes metrics that come from trees:
\begin{thm}\label{thm:fourpoint}
A dissimilarity map $\D$ is a tree metric if and only if for each $i,j,k,l \in X$, the value
\begin{equation}\label{eq:fourpoint}
\max\{\D(i,j)+\D(k,l), \D(i,k)+\D(j,l), \D(i,l)+\D(j,k)\}
\end{equation}
is realized by at least two of the three terms.
\end{thm}

Equation \ref{eq:fourpoint} is called the \emph{four-point condition}.
Theorem \ref{thm:fourpoint} motivates the development of algorithms
for identifying tree metrics that closely approximate dissimilarity maps obtained from
biological data. In molecular evolution, dissimilarity maps are derived by determining
distances between DNA sequences according to probabilistic models of
evolution; more generally, algorithms that approximate dissimilarity
maps by tree metrics can be applied to any distance matrices derived
from data. However, dissimilarity maps derived from data 
are never exact tree metrics due to two reasons: first, as discussed above, some
evolutionary mechanisms may not be realizable on trees; second, even
in cases where evolution is described well by a tree, finite sample sizes and random
processes underlying evolution lead to small deviations from ``treeness''
in the data.

It is therefore desirable to generalize the notion of a tree, and a
natural approach is to consider adding splits to those associated with
an $X$-tree. One natural class of split systems to consider is the following.
\begin{defn}\label{def:circular}
A \emph{circular ordering} $\C=\{x_1,\ldots,x_n\}$ is a bijection between $X$ and the vertices of a convex $n$-gon $P_n$
such that $x_i$ and $x_{i+1}$ map to adjacent vertices of $P_n$ (where $x_{n+1}:=x_1$).
Let $S_{i,j}$ denote the split\newline  $\{x_i, x_{i+1},\ldots,x_{j-1}\} | \{x_j,x_{j+1},\ldots,x_{i-1}\}$ and let $\S(\C) = \{S_{i,j} | i<j\}$. 
We say a split $S$ is \emph{circular} with respect to a circular ordering $\C$ if $S \in \S(\C)$, 
and a split system $\S$ is circular if $\S \subseteq \S(\C)$ for some circular ordering $\C$. 
\end{defn}
Given a set $\E$ of circular orderings, let $\S(\E)=\cap_{\C \in \E} \S(\C)$ be the system of splits that are circular
with respect to each ordering in $\E$. The split system associated
with a binary $X$-tree arises in this way from a family $\E$ of $2^{n-2}$ circular
orderings \cite{Semple2004}, and a tree metric is obtained by associating non-negative
weights to each split in the system. \emph{Kalmanson metrics}, which
were first introduced in the study of traveling salesmen problems 
where they provide a class of metrics for which the optimal tour can
be identified in polynomial time
\cite{Kalmanson1975}, correspond to the case when $|\E|=1$. 

Kalmanson metrics can be visualized using {\em split networks}
\cite{Huson2005b}. We do not provide a definition in this paper, but
show an example in Figure 2 (drawn using the software SplitsTree4).
\begin{figure}[!ht]
\label{fig2}
\includegraphics[scale=0.5]{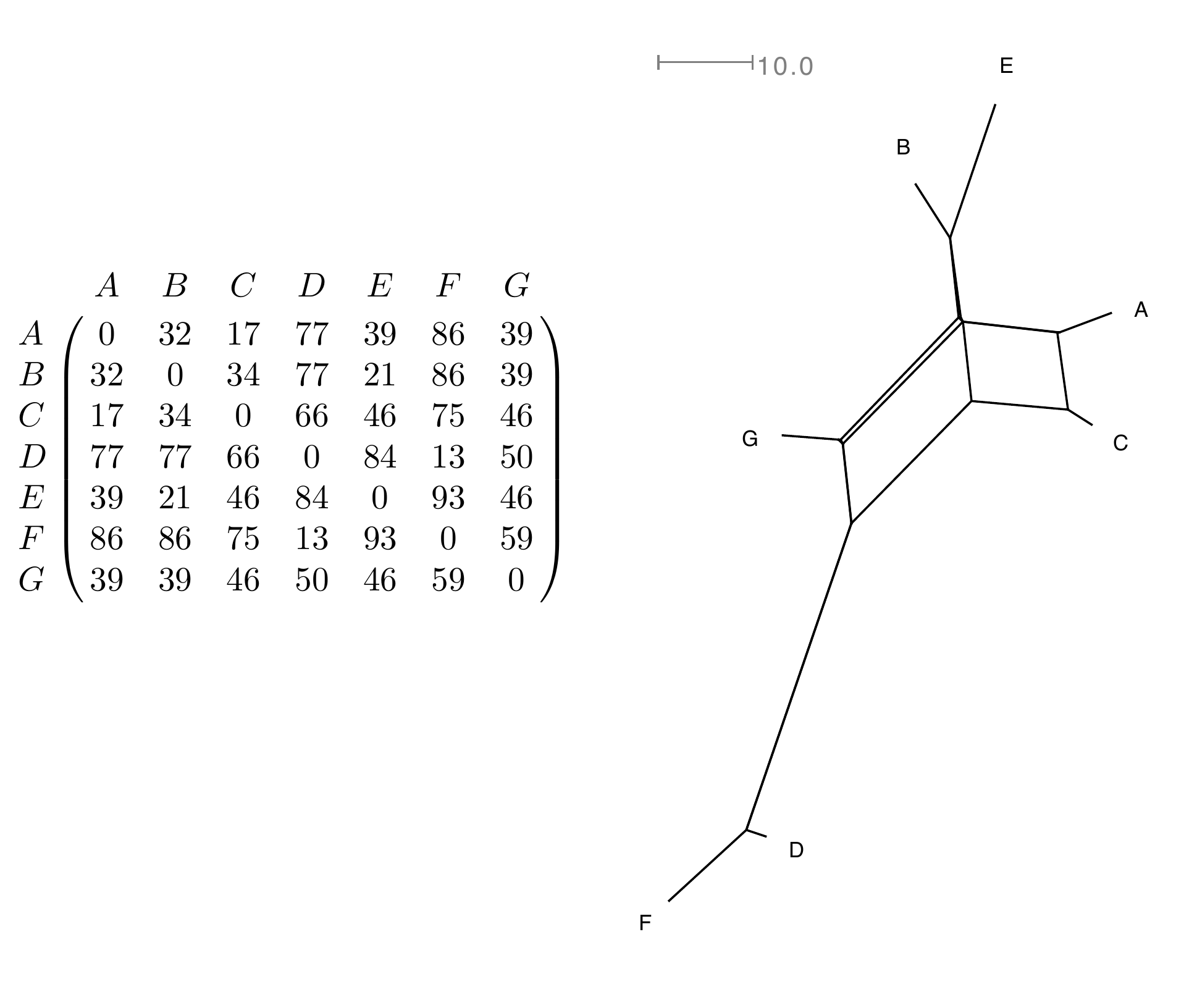}
\caption{A Kalmanson metric (left) visualized as a split network (right).}
\end{figure}
The neighbor-net \cite{Bryant2004} and MC-net
\cite{Eslahchi2010} algorithms provide a way to construct circular split
systems from dissimilarity maps, but despite having a number of useful
properties \cite{Bryant2007,Levy2010}, have not been widely adopted in
the phylogenetics community. This is likely because split networks (such as in Figure
2) fail to reveal the ``treeness'' of the
data. More specifically, the internal nodes of the split network do not
correspond to meaningful ``ancestors'' as do the internal nodes in $X$-trees.
Other approaches to visualizing ``treeness'', e.g. \cite{White2007}, do
reveal the extent of signal conflicting with a tree in data, but do
not reveal in detail the splits underlying the discordance. 

We propose that PQ- and PC-trees, first developed in the context of
the consecutive ones problem \cite{Booth1976,McConnell2003} and for graph
planarity testing \cite{Shih1999,Hsu2001}, are convenient
structures that interpolate between $X$-trees and full circular split
systems. The main result of this paper is Theorem
\ref{thm:bestfitPC}. Given a Kalmanson metric, the theorem shows how
to construct a ``best-fit'' PC-tree which realizes the metric and
captures the``treeness'' of it.

Another (expository) goal of this paper is to organize existing results on
PC-trees, their cousins PQ-trees, and corresponding metrics (Theorem \ref{thm:rectangle} in
Section \ref{sec:metrics}).
We believe this is the first paper to present the various results as
part of a single unified framework.  As a prelude, we illustrate the types of results
we derive using a classic theorem relating rooted $X$-trees to special set systems that encode information about shared
ancestry.
\begin{defn}\label{def:hierarchy}
A \em{hierarchy} $\H$ over a set $X$ is a collection of subsets of $X$ such that
\begin{enumerate}
\item $X \in \H$, and $\{x\} \in \H$ for all $x \in X$
\item $A \cap B \in \{\emptyset, A, B\}$ for all $A, B \in \H$.
\end{enumerate}
\end{defn}
The requirement that each $\{x\} \in \H$ is not part of the usual definition of hierarchy but its inclusion here will simplify later results.
We then have the following:
\begin{prop}\label{thm:zero}
There is a natural bijection between hierarchies over $X$ and rooted $X$-trees.
\end{prop}
Proposition \ref{thm:zero}, which we will prove constructively in
Section 2, is an elementary but classic result and has been discovered
repeatedly in a variety of
contexts \cite{Edmonds1977,Gusfield1991}. For example, in computer science, hierarchies are known as {\em laminar families}
where they play an important role in the development of recursive
algorithms represented by rooted trees (see, e.g. \cite{Fakcharoenphol2003}).
Hierarchies are also important because they are the combinatorial structures that underlie {\em ultrametrics}.
\begin{defn}\label{def:ultrametric}
An {\em ultrametric} is a symmetric function $\D: X \times X \to \mathbf{R}$ such that 
\[
\D(x,y) \leq \mathrm{max}\{\D(x,z), \D(y,z)\} \ \ \forall x,y,z \in X
\]
\end{defn}
\begin{defn}\label{def:indexed}
An {\em indexed hierarchy} is a hierarchy $\H$ with a non-negative function $f: \H \to \mathbf{R}_{\geq 0}$ such that for all $A,B \in \H$, $A \subset B \Rightarrow f(A) \leq f(B)$. 
\end{defn}
The extension of Proposition \ref{thm:zero} to metrics, proved in \cite{Jardine1967}, shows that these
objects are the same. The proposition is an instance of a {\em tree metric theorem} that
associates a class of combinatorial objects (in this case rooted $X$-trees)
with a class of metrics (in this case ultrametrics). Our results
organize other tree metric theorems that
have been discovered (in some cases independently and multiple times)
in the contexts of biology, mathematics and computer science.

In particular, we investigate relaxations of Definitions \ref{def:rootedXTree}, \ref{def:Sadditive} and \ref{def:hierarchy} for
which there exist analogies of Theorem \ref{thm:fourpoint} and
Proposition \ref{thm:zero}. For example,
hierarchies are special cases of {\em pyramids} \cite{Diday1986}, which can be
indexed to produce {\em strong  Robinsonian matrices} \cite{Robinson1951}. Proposition \ref{prop:RtoIP}
(originally proved in \cite{Diday1986}) states that these objects
correspond to each other mimicking the correspondence between hierarchies and
ultrametrics. 

In discussing tree metric theorems we adopt the nomenclature of Andreas
Dress who distinguishes two types of objects and theorems: the {\em
  affine} and the {\em projective} \cite{Dress1997}. Roughly speaking,
these correspond to ``rooted'' and ``unrooted'' statements
respectively, and we use these terms interchangeably. For example, a
hierarchy is an affine concept whose projective analog is a {\em
  pairwise compatible split system}.
Similarly, unrooted $X$-trees are the projective equivalents of
rooted $X$-trees, and tree metrics are the projective equivalents of
ultrametrics. We'll see that {\em Kalmanson metrics} are to tree
metrics as {\em Robinsonian matrices} \cite{Robinson1951} are to
ultrametrics, and circular split systems are to pairwise
compatible split systems as pyramids are to
hierarchies.
We use PQ-trees \cite{Booth1976}  and their projective analogs 
  PC-trees \cite{Shih1999} to link all of these results. 

\section{Hierarchies, $X$-trees and split systems}
We begin by proving Proposition \ref{thm:zero}, both for completeness and to introduce
some of the notation that we use. An $X$-tree $\T = (V, E, \phi)$ has a natural partial ordering on its vertices: 
for distinct $u,v\in V$, we say
$u \preccurlyeq v$ if $v$ lies on the unique path from $u$ to
the root. Given $v\in V$, let $H_v = \{x\in X |
x\preccurlyeq v \}$ and $\alpha(\T) = \{H_v | v\in V\}$.
\begin{prop}\label{prop:one}
The map $\alpha$ is a bijection from rooted $X$-trees to hierarchies over $X$.
\end{prop}
\begin{proof}
\verb"   "Let $(T,\phi)$ be a rooted $X$-tree with root $r$. $H_r = X$ 
and $H_{\phi(x)} = \{x\}$ for all $x \in X$, so $\H$ satisfies (1) of Definition \ref{def:hierarchy}.
Consider any two $H_u, H_v \in \alpha(\T)$. If $u \preccurlyeq v$ then $H_u \cap H_v = H_u$,
if $v \preccurlyeq u$ then $H_u \cap H_v = H_v$, and otherwise $H_u \cap H_v = \emptyset$.
So each pair of elements in $\alpha(\T)$ satisfies (2) of Definition \ref{def:hierarchy} and $\alpha(\T)$ is a hierarchy.

For the reverse direction, let $\H$ be a hierarchy over $X$. Let $T=(V,E)$ be the digraph with $V = \{v_A | A \in \H\}$ and with edges denoting minimal inclusion:
$T$ has an edge from $v_B$ to $v_A$ iff
$A \subset B$ and there does not exist $C \in \H$ such that $A \subsetneq C \subsetneq B$.
We will show that $T$ is a tree. First note that by induction on $|C|$ each vertex $v_C$ with $C \neq X$ is connected to $v_X$ and has at least one parent.
Now suppose $v_{A}, v_{B}$ are distinct parents of $v_C$.
Then $A \cap B \neq \emptyset$, so without loss of generality  by the hierarchy condition $A \subset B$.
But then $C \subset A \subset B$, a contradiction. Thus $T$ is connected and has one fewer edge than vertices, so $T$ is a tree with root $v_X$.
Define the map $\phi: X \to V$ by $\phi_{\{x\}} = v_x$. This is a bijection from $X$ to the leaves, so 
$\T=(T,\phi)$ is a rooted $X$-tree with $\alpha(\T) = \H$.
\end{proof}
The above proposition gives a characterization of rooted $X$-trees in terms of collections of subsets of $X$. 
We turn now to the projective analogue of rooted $X$-trees. 
\begin{defn}\label{def:split2}
Two $X$-splits $A_1|B_1, A_2|B_2$ are \emph{compatible} if one of $A_1 \cap A_2, A_1 \cap B_2, B_1 \cap A_2, B_1 \cap B_2$ is
empty, and are \emph{incompatible} otherwise. 
\end{defn}

Removing an edge $e$ from a projective $X$-tree disconnects the tree
into two pieces and thus gives an $X$-split $S_e$. Let $\beta(\T) =
\{S_e | e \in E(T)\}$. It is easy to check that $\beta(\T)$ is a
pairwise compatible split system, and a basic theorem of $X$-trees
\cite{Buneman1971}\cite{ASCB2005} shows every pairwise compatible split system arises in this way, giving
\begin{prop}\label{thm:splitsequiv}
$\beta$ is a bijection from the set of projective $X$-trees to the set
of pairwise compatible split systems over $X$.
\end{prop}
Finally we show that pairwise compatible split systems and hierarchies are in bijection.
Fix $r\in X$ and let $\mathcal{S}$ be a set of pairwise compatible splits
of $X$.
The \emph{unrooting map} $\gamma_r$ sends a split $S$ to the component of $S$ that does not contain $r$.
The \emph{rooting map} $\delta_r$ sends a set $A \subseteq X \setminus \{r\}$ to the split $\delta_r(A) = X | X \setminus A$.
If $\S$ is a split system, let $\gamma_r(\S)=\{\gamma_r(S)| S \in \S\}$.
\begin{prop}
$\mathcal{S}$ is a pairwise compatible split system over $X$ if and only if $\gamma_r(\S)$ is a hierarchy over $X \setminus \{r\}$.
\end{prop}
\begin{proof}
Choose $S_1, S_2 \in \mathcal{S}$, with $S_i = A_i | B_i$. If $S_1,
S_2$ are compatible,
then without loss of generality  $A_1 \cap A_2 = \emptyset$. If $r \in
B_1,B_2$, then $\gamma_r(S_i) = A_i$, so $\gamma_r(S_1) \cap
\gamma_r(S_2) = \emptyset$.
If $r \in A_1$, then $r \not \in A_2$ since $A_1 \cap
A_2 = \emptyset$, and therefore $r \in B_2$. In this case $\gamma_r(S_1) = X \setminus
A_1$ and $\gamma_r(S_2) = A_2$, so $\gamma_r(S_1) \cap
\gamma_r(S_2) = A_2 = \gamma_r(S_2)$, again satisfying the hierarchy
condition. The final case follows by symmetry. Conversely, suppose $\gamma_r(\S)$ is a hierarchy.
Then for any two $S_1, S_2 \in \S$ we have $\gamma_r(S_1) \cap \gamma_r(S_2) \in \{\emptyset, \gamma_r(S_1), \gamma_r(S_2)\}$.
Checking the cases as above shows that $S_1$ and $S_2$ must
then be compatible splits.
\end{proof}
We define the map $\kappa_r$ to take an affine $X \setminus \{r\}$-tree $\T$ to the projective $X$-tree obtained by attaching a vertex with label $r$ to the root of $\T$.
The inverse map $\lambda_r$ takes a projective $X$-tree $\T$ to the affine $X$-tree as follows: let $v$ be the vertex of $\T$ labelled $r$. 
Then $\lambda_r(\T)$ is obtained by rooting at the neighbor of $v$, and then deleting $v$.
\begin{prop}\label{thm:rect1}
If AT and H are the sets of all affine trees and
hierarchies over $X \setminus \{r\}$, respectively, and PT and PSS are the
sets of all projective trees and pairwise compatible splits
systems over $X$, respectively, then the following diagram commutes:
\[
\xymatrix{
AT  \ar@<.5ex>[d]^\alpha  \ar@<.5ex>[r]^{\kappa_r}  & PT \ar@<.5ex>[d]^\beta \ar@<.5ex>[l]^{\lambda_r} \\
H   \ar@<.5ex>[r]^{\delta_r} \ar@<.5ex>[u] & PSS \ar@<.5ex>[l]^{\gamma_r} \ar@<.5ex>[u] }
\]
Each arrow is a bijection; the unlabelled arrows are the inverses of the maps going in the other direction.
\end{prop}
\section{PQ-trees}\label{sec:PQ}
We start our generalization of Proposition \ref{thm:rect1} with a generalization of rooted $X$-trees.
\begin{defn}\label{def:PQtree}
A {\em PQ-tree} over $X$ is a rooted $X$-tree in which every vertex comes equipped with a linear ordering on its children.  Every internal vertex of degree three or less is labeled a P-vertex, and every internal vertex of degree four or more is labeled either as a P-vertex or a Q-vertex. 
We say two PQ-trees $\T_1, \T_2$ are equivalent (we write $\T_1 \sim \T_2$) if one can be obtained from the other
by a series of moves consisting of:
\begin{enumerate}
\item Permuting the ordering on the children of a P-vertex,
\item Reversing the ordering on the children of a Q-vertex.
\end{enumerate}
\end{defn}
We draw a PQ-tree by representing P-vertices as circles and Q-vertices as squares,
and ordering the children of a vertex from left to right as per the
corresponding linear order (see Figure 3).
\begin{figure}[!h]\label{fig3}
\begin{center}
\includegraphics[scale = 0.5]{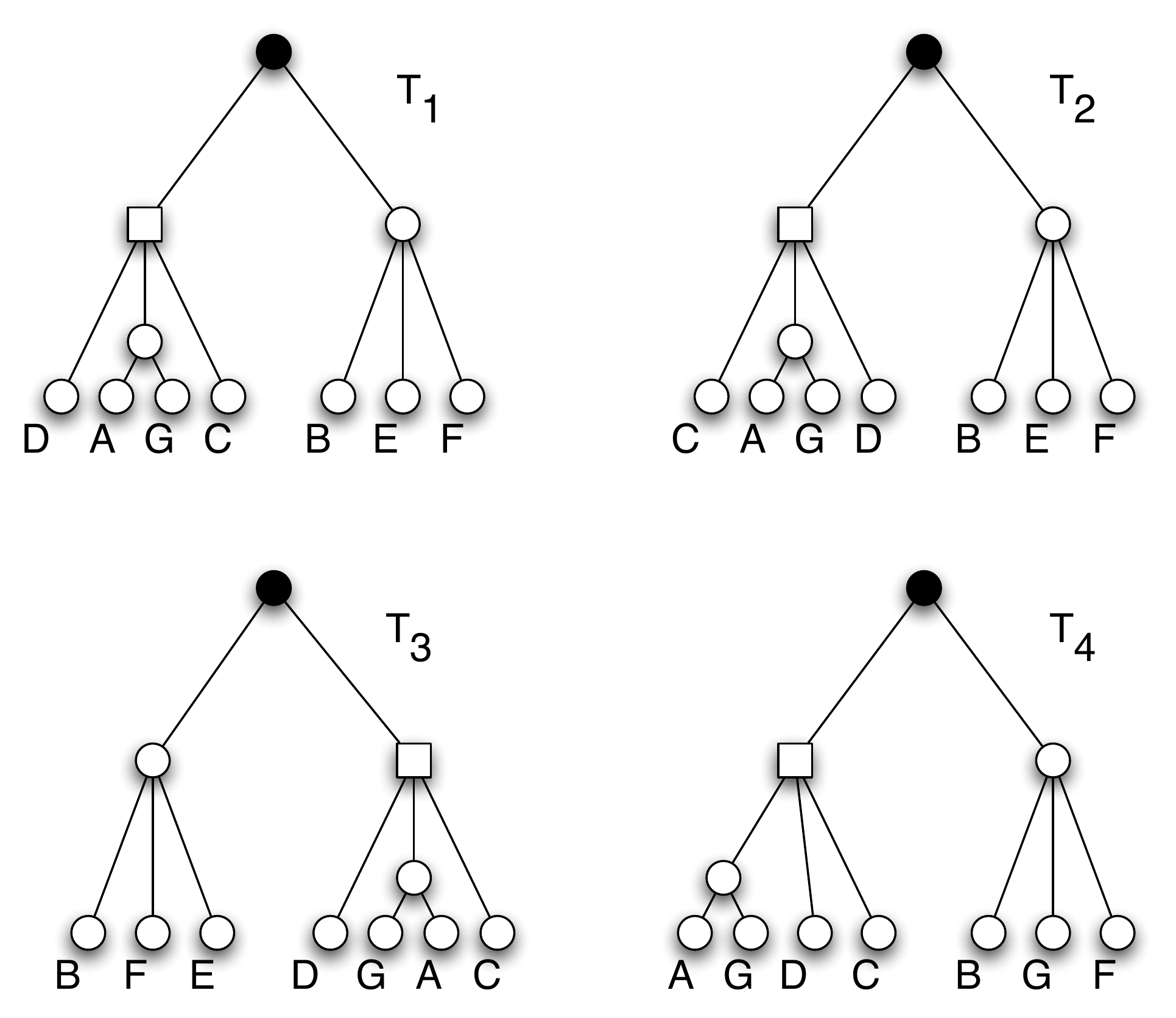}
\end{center}
\caption{
Four PQ trees. $T_1,T_2$ and $T_3$ are equivalent to each other but $T_4$ is different from the other three.
}
\end{figure}
For any PQ-tree $\T$ over $X$, define the
$\textit{frontier}$ of the tree as the linear ordering on $X$ derived from
reading the leaves of $T$ from left to right. Let $con(\T) =
 \{frontier(\T')| \T' \sim \T\}$ be the set of all linear
orderings $\prec$ that are consistent with the PQ structure of $\T$. 
We say $I$ is an \emph{interval} with respect to 
$\prec$ if there exist $a,b \in X$ such that $I = \{ t | a
\preceq t \preceq b\}$.
Define $\alpha$ to be the map that sends the PQ-tree $\T$ to the set of all $I \subseteq X$ such
that $I$ is an interval with respect to every linear ordering in $con(A)$.
\begin{lemma}\label{lem:1}
$\alpha(\T)$ is a hierarchy if and only if every vertex of $\T$ is a P-vertex. 
If so, let $\T'$ be the corresponding normal affine $X$-tree. Then
$\alpha(\T) = \alpha(\T')$, where $\alpha(\T')$ is the hierarchy constructed in Proposition \ref{prop:one}.
\end{lemma}
\begin{proof}
Let $v$ be an internal vertex of $T$, $\{c_1, c_2, \dots, c_n\}$ 
the set of its children, and recall $H_v = \{x \in X | v_x \preceq v\}$ 
is the set of all the elements $x$ such that the path from $v_x$ to the root 
includes $v$. Now $H_{c_1} \cup \cdots \cup H_{c_n}$ is in $\alpha(\T)$, 
and if every vertex of $A$ is a $P$-vertex then every element of $\alpha(\T)$ 
will be of this form. In this case $\alpha(\T)$ is identical to 
the hierarchy constructed in Proposition \ref{prop:one}. Now suppose $T$ has a $Q$-vertex $v$. 
Then $n \geq 3$ and both $A=H_{c_1}\cup H_{c_2}$ and $B=H_{c_2} \cup H_{c_3}$ 
are in $\alpha(\T)$, but $A \cap B = H_{c_2}$ is nonempty, 
so $\alpha(\T)$ is not a hierarchy. 
\end{proof}
This shows that the map $\alpha$ on PQ-trees agrees with the $\alpha$
in Proposition \ref{thm:rect1}. 
It also shows that the usual affine $X$-trees are precisely PQ-trees with all P-vertices. Since PQ-trees do not necessarily give rise to hierarchies,
we seek a different combinatorial characterization of them.
\begin{defn}\label{def:pyrHierarchy}
A collection of subsets $\P$ of $X$ is a \emph{prepyramid} if
\begin{enumerate}
\item $X \in \P$ and $\{x\} \in \P$ for all $x \in X$,
\item There exists a linear ordering $\prec$ on $X$ such that every $A \in \P$ is an interval with respect to $\prec$. 
\end{enumerate}
$\P$ is a \emph{pyramid} if, in addition, it is closed under intersection.
\end{defn}
If $\T$ is a PQ-tree then $\alpha(\T)$ is a prepyramid with respect
to any $\prec \in frontier(\T)$, and $\alpha(\T)$ is the \emph{associated prepyramid}. 
\begin{defn}\label{def:rootedFam}
Two subsets $A,B$ of $X$ are $\emph{compatible}$ if $A \cap B \in \{\emptyset, A, B\}$. Otherwise they are \emph{incompatible}
A $\textit{rooted family}$ over $X$ is a collection of sets $\F$ such that
if $A,B \in \F$ are incompatible, then $A\cap B, A \setminus B,B \setminus A$, and $A \cup B$ are in $\F$.
\end{defn}
We are now ready to state the main result of this section.
\begin{prop}\label{thm:PQconstruct}
The map $\alpha$ is a bijection from PQ-trees to prepyramids that are rooted families.
\end{prop}
\begin{proof}
If $\T$ is a PQ-tree, $\alpha(\T)$ is obviously a prepyramid. We now show it is also a rooted family.
Trivially $X \in \alpha(\T)$ and $\{x\} \in \alpha(\T)$ for all $x \in X$. 
Let $A,B$ be incompatible sets in $\F$ and $\prec \in con(A)$ a linear ordering.
We can write $A=\{t|x_A \preceq t \preceq y_A\}$ and $B = \{t|x_B \preceq t \preceq
y_B\}$ for some $x_A,x_B,y_A,y_B \in X$, and by incompatibility we can assume 
 $x_A \prec x_B \preceq y_A \prec y_B$. Then $A \cap B = \{t|x_B
\preceq t \preceq y_A\}$, $A \cup B = \{t|x_A \preceq t
\preceq y_B\}$, $A \setminus B = \{t| x_A \preceq t \prec x_B \}$ and $B \setminus A = \{t| y_A \prec t \preceq y_B\}$.
Since each of these four sets is an interval with respect to $\prec$ for every $\prec \in con(A)$ they
are each in $\alpha(\T)$.

It remains to show that if $\F$ is a collection of subsets of $X$ which is a prepyramid and a rooted family,
then there exists a unique PQ-tree $\T$ such that $\alpha(\T) = \F$.
Let $\F' \subseteq \F$ consist of the sets in $\F$ that are compatible with all of $\F$.
Then the elements of $\F'$ are pairwise compatible, so $\F'$ is a hierarchy and corresponds to a tree $\T'$, as in Proposition \ref{prop:one}. 
Consider the vertices $v_C$ in $T'$ corresponding to subsets of the form $C=A \cup B$ with $A, B \in \F$ incompatible, 
and mark those vertices as $Q$-vertices.
For each such $v_C$ there is a natural ordering on its children $v_{C_1}, v_{C_2}, \ldots, v_{C_n}$ as follows: 
$v_{C_i} < v_{C_j}$ if the labels of the leaves in the subtree rooted at $v_{C_i}$ 
are all $\prec$ the labels of the leaves of the subtree rooted at $v_{C_j}$,
where $\prec$ is the order from the rooted family condition.
Ordering the children of the $Q$-vertices of $T'$ in this way,
we obtain a PQ-tree $\T$.

We will use induction on $|\F|$ to show $\alpha(\T)=\F$ and that $\T$ is the only such PQ-tree for which this is true.
First, suppose $\F'$ contains a set $A \neq X$, $|A|>1$.
Define $\F_1 = \{C \in \F | A \subseteq C \mathrm{\ or\ } A \cap C = \emptyset\}$ and $\F_2 = \{C \in \F | A \supseteq C\}$. 
$\F_1 \cap \F_2 = A$, and because $A$ is compatible with everything $\F_1 \cup \F_2 = \F$.
$\F_1$ is a prepyramid over $(X \setminus A) \cup \{A\}$and $\F_2$ is a prepyramid over $A$, and both are rooted families 
with $|\F_i| < |\F|$. Then the inductive hypothesis shows there are PQ-trees $\T_1, \T_2$ such that 
$\alpha(\T_i)=\F_i$ for $i=1,2$. Now $\T_1$ has a leaf corresponding to $A$, and the root of $\T_2$ also corresponds to $A$.
Grafting $\T_2$ onto the leaf $A$ in $\T_1$ gives a PQ-tree $\T$ with $\alpha(\T) = \alpha(\T_1) \cup \alpha(\T_2) = \F$.
Conversely, if $\T$ is a PQ-tree with $\alpha(\T)=\F$, 
$\T$ must have a node $v$ such that the subtree $\T_1$ with root $v$ satisfies $\alpha(\T_1)=\F_1$,
and the tree $\T_2$ obtained by replacing the subtree $\T_1$ with the vertex $v$ and label $A$
gives $\alpha(\T_2)=\F_2$. The inductive hypothesis gives the uniqueness of $\T_1$ and $\T_2$, 
which in turn implies the uniqueness of $\T$.

Now suppose no such set $A$ exists, so $\F'=\{\{x_1\},\ldots,\{x_n\},X\}$. If $\F=\F'$ then $\F$ is a hierarchy
and the PQ-tree of depth one with root a P-vertex
is the unique tree such that $\alpha(\T)=\F$. We now consider the final case, where $\F$ contains sets 
other than $X$ and the $\{x\}$s, and every element in $\F$ except for these is incompatible
with another element of $\F$. Without loss of generality  assume $x_1 \prec x_2 \prec \cdots \prec x_n$, where $\prec$
is the ordering under which $\F$ is a prepyramid. Then every element of $\F$ 
is of the form $x_{[i:j]}:=\{x_i, x_{i+1}, \cdots, x_j\}$ for $1 \leq i \leq j \leq n$. 
Let $\T$ be the PQ-tree of depth $1$ with a Q-vertex root and children
$x_i$. Since $\alpha(\T) = \{x_{[i:j]} | 1 \leq i \leq j \leq n\}$ it's clear $\alpha(\T) \supseteq \F$.

We must show $\F = \alpha(\T)$, or equivalently, that $x_{[i:j]} \in \F$ for all $i < j$.
We will prove this by induction on $n$. It's obvious for $n=3$, so suppose $n\geq 4$.
Let $A \neq X$ be a set in $\F$ that is maximal under inclusion.
By assumption $|A| > 1$ and $A$ is incompatible with some $B \in \F$. By the rooted family
condition $A \cup B \in \F$ so we must have $A \cup B = X$ by the 
maximality of $A$. Without loss of generality  we may write $A = x_{[1:j]}$,
$B = x_{[i:n]}$ with $i \leq j$. We claim $j = n-1$.
For by the rooted family condition $\F$ contains
$A \setminus B = x_{[1:i-1]}$, $A \cap B = x_{[i:j]}$ and $B \setminus A = x_{[j+1:n]}$.
If $j \neq n-1$, $B \setminus A$ is incompatible with some other set $C \in \F$.
$C$ must be of the form $C = x_{[k_1:k_2]}$ with $k_1 < j+1 \leq k_2 < n$.
If $k_1=1$ then $C \supset A$ contradicting the maximality of $A$.
Then $k_1\neq 1$, so $A$ and $C$ are incompatible
and $A \cup C = x_{[1:k_2]} \in \F$, contradicting the maximality of $A$.
Thus $j=n-1$ and $x_{[1:n-1]} \in \F$. By the same reasoning, $x_{[2:n]} \in \F$.

Now let $\G = \{A \in \F | A \subseteq x_{[1:n-1]}\}$.
Since $x_{[1:n-1]} \in \G$, $\G$ is a prepyramid over $X \setminus \{x_n\}$.
It is also a rooted family. We claim that every element $A \in \G$ with
$A \neq \{x_i\}, A \neq X  \setminus  \{x_n\}$ is incompatible with some other element of $\G$.
To see this, write $A = x_{[i:j]}$, $j \leq n-1$ and suppose $j \neq n-1$. 
By our initial assumption $A$ is incompatible with some set 
$B = x_{[i':j']} \in \F$. If $j' \neq n$ then $B \in \G$ gives the incompatible set.
Otherwise $B$ and $x_{[1:n-1]}$ are incompatible, so $B \cap x_{[1:n-1]} = x_{[i':n-1]}$
is in $\G$ and is incompatible with $A$.
Finally, suppose $j = n-1$.
Then $A$ is incompatible with $x_{[n-1:n]}$, so $A \cup x_{[n-1:n]} = x_{[i:n]} \in \F$.
This must be incompatible with some other set $B \in \F$.
We can write $B=x_{[i':j']}$, $i \leq j' \leq n-1$, so  
$B \in \G$ is incompatible with $A$ and the claim is proved.

It follows that $\G$ is a rooted prepyramid over $X \setminus \{x_n\}$,
and every element in $\G$ other than $X$ and the $\{x_i\}$s is 
incompatible with some other element of $\G$. 
By our inductive hypothesis, $x_{[i:j]} \in \G$
for each $1 \leq i \leq j \leq n-1$.
We showed that $x_{[n-1:n]}\in \F$, and for each $i$ the sets
$x_{[i:n-1]}$  and $x_{[n-1:n]}$ are incompatible,
so their union $x_{[i:n]}$ is in $\F$.
Thus $\F = \{x_{[i:j]}|1 \leq i \leq j \leq n\}$.
\end{proof}
\section{PC-trees}
We now describe the projective equivalent of PQ-trees.
\begin{defn}\label{def:PCtree}
A \emph{PC-tree} over $X$ is an $X$-tree where each internal vertex 
comes equipped with a circular ordering of its neighbors. Additionally, each internal vertex
of degree less than $4$ is labelled a P-vertex, and each other vertex is labelled
either a P-vertex or a C-vertex (Figure 4).
Two PC-trees $\T_1, \T_2$ are said to be equivalent (we write $\T_1 \sim \T_2$)  if one can be obtained from the other by a series of the following moves:
\begin{enumerate}
\item Permuting the circular ordering of the neighbors of a P-vertex,
\item Reversing the circular ordering on the neighbors of a C-vertex.
\end{enumerate}
\end{defn}
\begin{figure}\label{fig4}
\begin{center}
\includegraphics[scale=0.5]{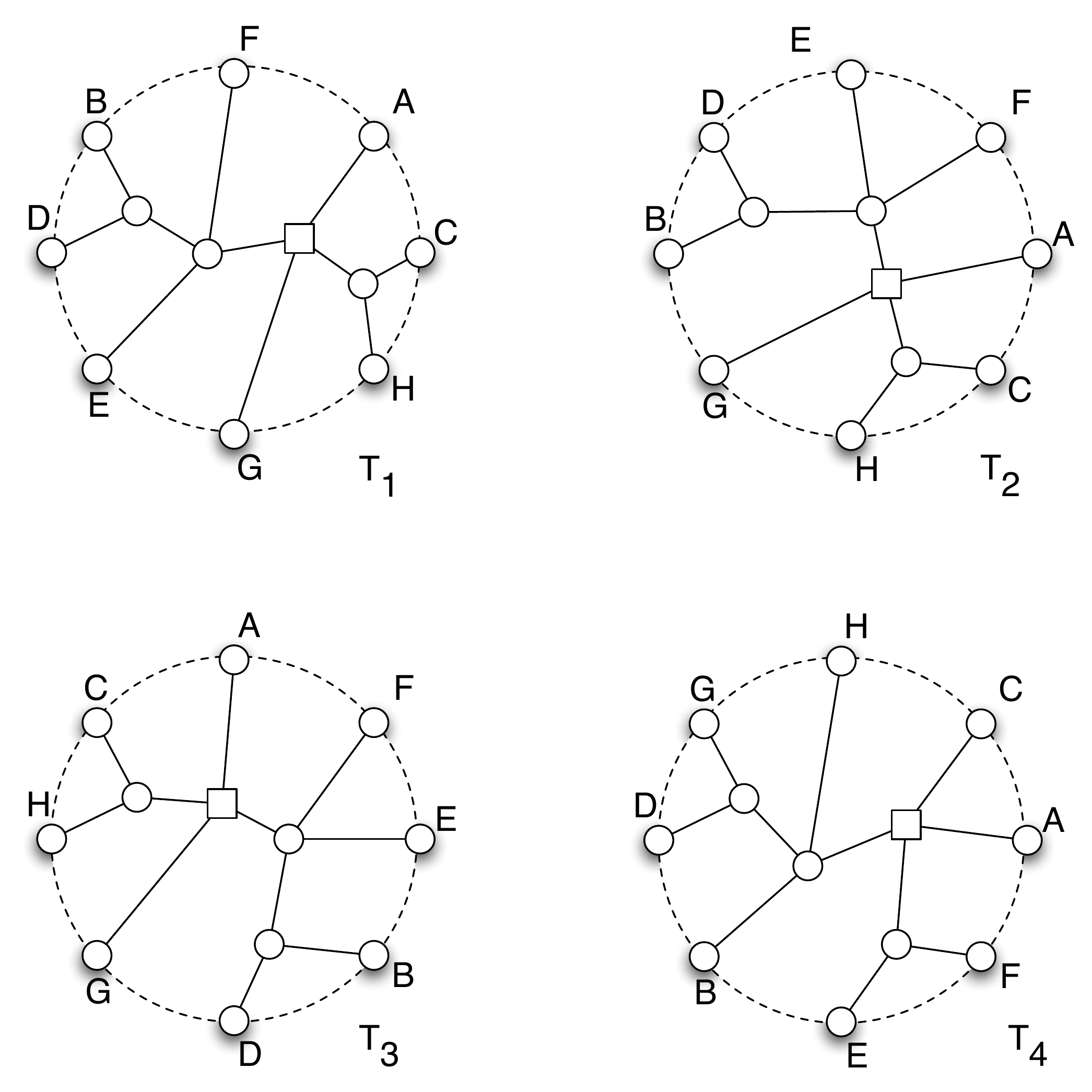}
\end{center}
\caption{Four PC-trees. $T_1,T_2$ and $T_3$ are
  equivalent to each other but $T_4$ is different from the other
  three.}
\end{figure}

For a PC-tree $\T$, let $\textit{frontier}(\T)$ be the circular
ordering given by reading the taxa in either a clockwise or
counterclockwise direction. Let $con(\T) = \{frontier(\T') | \T' \sim \T\}$ and
let  $\beta(\T) = \cap_{\C \in con(\T)} \S(\C)$.
We say $\beta(\T)$ is the circular split system \emph{associated} to $\T$. 
\begin{defn}
A split system $\S$ is an \emph{unrooted split family over $X$} if, for each pair
of incompatible splits $S_1 =A_1|B_1, S_2=A_2|B_2$ in $\S$,
the splits $A_1\cap A_2 | B_1 \cup B_2, A_1 \cap B_2 | A_2 \cup B_1,
A_2 \cap B_1 | A_1 \cup B_2$, and $B_1 \cap B_2 | A_1 \cup A_2$
are all in $\S$ as well.
\end{defn}
Following the proof in Proposition \ref{thm:PQconstruct} that $\alpha(\T)$ is a rooted family for every PQ-tree $\T$, one can show $\beta(\T)$ is an unrooted family for every PC-tree $\T$.
\begin{lemma}\label{lem:2}
For any PC-tree $\T$, $\beta(\T)$ is an unrooted split family.
\end{lemma}
We now generalize the map $\gamma_r$ to PQ- and PC-trees.
\begin{defn}
The \emph{unrooting map} $\kappa_r$ sends a PQ-tree $\T$ over $X
\setminus \{r\}$ to the PC-tree $\kappa_r(\T)$ as follows: attach a vertex labelled $r$ to the root of $\T$. 
If vertex $v$ in $\T$ has children $\{v_1,v_2,\ldots,v_k\}$ with linear ordering $v_1 \prec v_2 \prec \ldots \prec v_k$ and parent $w$, in $\kappa_r(\T)$
the vertex has the same neighbors with circular ordering $\{v_1,\ldots,v_k,w\}$. 

The \emph{rooting map} $\lambda_r$ sends a PC-tree $\T$ over $X$ to the 
PQ-tree over $X \setminus \{r\}$ obtained by rooting at the vertex adjacent to $r$, 
deleting $r$, and replacing each $C$ vertex with a $Q$ vertex. 
Let $v$ be such a vertex with a circular ordering $\C = \{v_1, \ldots, v_m\}$; 
we may assume the path from $v$ to the root passes through $v_m$. 
Then in $\lambda_r(\T)$, vertex $v$ has parent $v_m$ and children $v_1,\ldots,v_{m-1}$ with linear ordering $v_1 \prec v_2 \prec \ldots \prec v_{m-1}$.
\end{defn}
Since $\kappa_r$ and $\lambda_r$ are inverses, this immediately gives the following:
\begin{prop}\label{lem:klBij}
$\kappa_r$ is a bijection from PQ-trees over $X \setminus \{r\}$ to PC-trees over $X$.
\end{prop}
Recall that if $S$ is a split of $X$, the map $\gamma_r$ takes $S$ to the component of $S$ not containing $r$.
\begin{prop}\label{lem:3}
Let $CUF$ be the set of circular split systems that are unrooted families over $X$,
and let $PP$ be the set of prepyramids that are rooted families over $X \setminus \{r\}$.
Then the map $\gamma_r$ is a bijection from $CUF$ to $PP$.
\end{prop}
\begin{proof}
$\delta_r$ and $\gamma_r$ are inverses, so it suffices to show that 
$\gamma_r(CUF) \subseteq PRF$ and $\delta_r(PRF) \subseteq CUF$. Let $\S$ be a circular unrooted split family
with circular ordering $\{x_1,\ldots,x_n\}$ and suppose $r = x_i$. Then $\gamma_r(\S)$ is a prepyramid
with respect to the linear ordering on $X \setminus \{r\}$ given by 
$x_{i+1} \prec x_{i+2} \prec \cdots \prec x_{i-1}$.

Next, consider $\S\in CUF$ and two incompatible sets $G,H \in \gamma_r(\S)$,
with $G = \gamma_r(S_1)$ and $H = \gamma_r(S_2)$. Assume $S_1= A_1|B_1$, 
$S_2 = A_2 | B_2$ are compatible as split systems
with $A_1 \cap A_2 = \emptyset$. If $r \in
A_1$, then $r \in B_2$ and $G\cap H = (X-A_1) \cap A_2 = A_2 = H$,
contradicting the incompatibility of $G$
and $H$. The other cases produce similar contradictions, so $S_1$ and
$S_2$ must be incompatible as split systems. Then the splits
$A_1 \cap A_2 | B_1 \cup B_2, A_1 \cap B_2 | A_2 \cup B_1, A_2 \cap B_1 | A_1 \cup B_2$
and $B_1 \cap B_2 | A_1 \cup A_2$ are all in $\S$.
Assume without loss of generality  that $G = A_1, H = A_2$. Then
$\gamma_r(\{A_1\cap A_2, B_1\cup B_2\}) = A_1 \cap A_2 = G \cap H \in \gamma_r(\S)$,
and similarly $G \cap H, G \setminus H, H \setminus G$ are all in $\gamma_r(\S)$,
so $\gamma_r(\S)$ is a rooted family.

Conversely, let $\F$ be a pyramidal, rooted family over $X \setminus \{r\}$ with linear ordering $\prec$. 
$\delta_r(\F)$ is circular and contains all the trivial splits, as $\{x\} | X \setminus \{x\} = \gamma_r(\{x\})$ and $\{r\} | X \setminus \{r\} = \gamma_r(X \setminus \{r\})$. 
The above argument reverses to show $\delta_r(G), \delta_r(H) \in \delta_r(\F)$ are incompatible as split systems only if $G$ and $H$ are incompatible as sets.
In this case $\delta_r(G \cap H), \delta_r(G \cup H), \delta_r(G \setminus H)$ and $\delta_r(H \setminus G)$ are in $\delta_r(\F)$ so $\delta_r(\F)$ 
is an unrooted family. Thus, $\delta_r(PRF) = CUF$ and $\gamma_r(CUF) = PRF$ as required.
\end{proof}
The above propositions combine to show that the map $\kappa_r \circ \alpha^{-1} \circ \gamma_r$ is a bijection from circular split systems to PC-trees, 
and that in fact for a PC-tree $\T$, $\kappa_r \circ \alpha^{-1} \circ \gamma_r(\T)$ is precisely the split system arising from $\T$ in the natural way. 
An example is shown in Figure 5. We thus have:
\begin{prop}\label{thm:rect2}
The following diagram commutes:
\[ \xymatrix{
PQ  \ar@<.5ex>[d]^\alpha  \ar@<.5ex>[r]^{\kappa_r}  & PC \ar@<.5ex>[d]^\beta \ar@<.5ex>[l]^{\lambda_r} \\
PRF   \ar@<.5ex>[r]^{\delta_r} \ar@<.5ex>[u] & CUF \ar@<.5ex>[l]^{\gamma_r} \ar@<.5ex>[u]
} \]
where $PQ$ is the set of PQ-trees over $X \setminus \{r\}$, PRF is the set of prepyramids that are rooted families over $X \setminus \{r\}$, 
$PC$ is the set of PC-trees over $X$, and $CUF$ is the set of circular unrooted families over $X$.
\end{prop}
\begin{figure}\label{fig5}
\begin{center}
\includegraphics[scale=0.5]{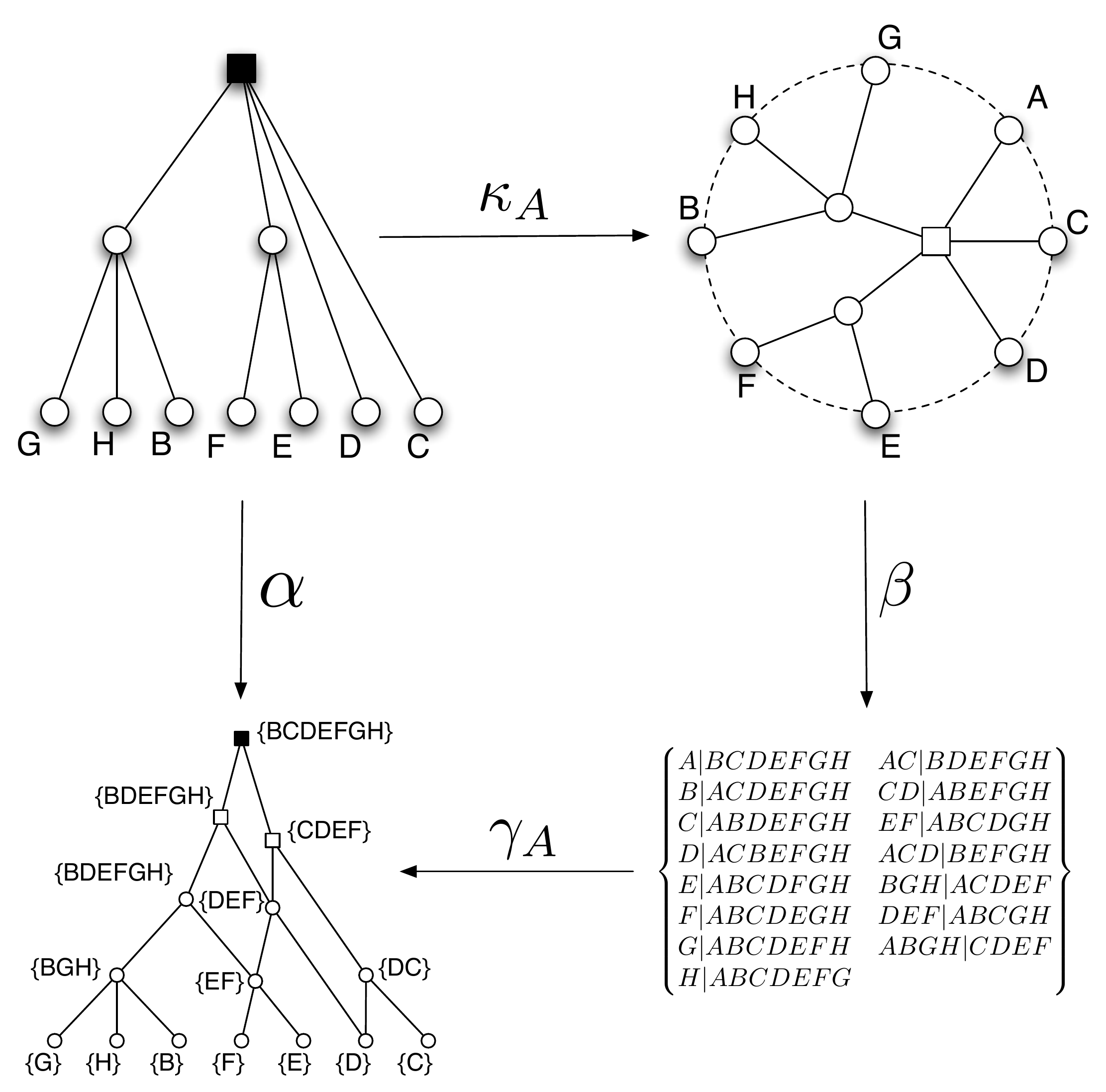}
\end{center}
\caption{
An instance of Proposition \ref{thm:rect2}.}
\end{figure}

\newpage
\section{Metrics Realized By PC-trees}\label{sec:metrics}
In this section we extend the tree metric theorem (Theorem
\ref{thm:fourpoint}) by showing how to replace affine and projective
trees with PQ- and PC-trees. To explain our approach, we recall
that in Section 2 we reviewed the connection between trees and their
representations as hierarchies and split systems:
\[ \xymatrix{
AT  \ar@<.5ex>[d]^\alpha  \ar@<.5ex>[r]^{\kappa_r}  & PT \ar@<.5ex>[d]^\beta \ar@<.5ex>[l]^{\lambda_r} \\
H   \ar@<.5ex>[r]^{\delta_r} \ar@<.5ex>[u] & PSS \ar@<.5ex>[l]^{\gamma_r} \ar@<.5ex>[u]
} \]
This correspondence can be extended to metrics as follows \cite{Semple2003}:
\[
\xymatrix{
AT \ar@<.5ex>[d]^\alpha  \ar@<.5ex>[r]^{\kappa_r}  & PT \ar@<.5ex>[d]^\beta \ar@<.5ex>[l]^{\lambda_r} \\
H   \ar@<.5ex>[r]^{\delta_r} \ar@<.5ex>[u] & PSS \ar@<.5ex>[l]^{\gamma_r} \ar@<.5ex>[u] \\
U \ar@<.5ex>[u] \ar@<.5ex>[r] & TM \ar@<.5ex>[u] \ar@<.5ex>[l] }\]
Here $U$ are ultrametrics and $TM$ are tree metrics. The tree metric
theorem is proved by diagram chasing: starting with a tree metric, the
Gromov product is applied (see Definition \ref{def:gromov} below),
resulting in an ultrametric. A (unique) hierarchy representing the
ultrametric can be obtained and then the PSS corresponding to the
hierarchy can be derived by the unrooting map $\delta_r$. This weighted PSS
represents the tree metric. 

In this section we construct a PC-tree that best realizes a Kalmanson metric by a similar approach, 
constructing an analog of the above diagram with suitable PQ- and PC-tree counterparts (Theorem \ref{thm:rectangle}). 
The extension requires some care, because the weighted PC-trees representing a Kalmanson metric may require extra zero splits. 
A key result (Theorem \ref{thm:bestfitPC}) is that there is a unique PC-tree that minimally represents any Kalmanson metric.

We begin by making precise the notion of a Kalmanson metric.
\begin{defn}\label{def:kalmanson}
A dissimilarity map $\D$ is $\emph{Kalmanson}$ if there is a circular ordering $\{x_1,\ldots,x_n\}$ such that for for all $i<j<k<l$,
\begin{equation}\label{eq:kal}
\max\{\D(x_i,x_j) + \D(x_k,x_l), \D(x_l,x_i) + \D(x_j,x_k)\} \leq \D(x_i,x_k) + \D(x_j,x_l).
\end{equation}
\end{defn}
Let $\T$ be a projective $X$-tree and $\C$ a circular ordering obtained by reading the taxa of $\T$ clockwise. If $D$ is $\T$-additive then $D$ is Kalmanson with respect to $\C$. Additionally, in this case one actually has equality in \eqref{eq:kal} for each $i<j<k<l$. Kalmanson metrics are thus generalizations of tree metrics obtained by relaxing the equality conditions of the four-point theorem. 
The following theorem, proved in \cite{Chepoi1998}, gives the Kalmanson metric equivalent
of the four-point condition.
\begin{thm}\label{thm:circSplitsEquiv}
A metric $\D$ satisfies the Kalmanson condition if and only if
there exists a circular split system $\S$ and weight function $w: \S \to \mathbf{R}^+$
such that $\displaystyle \D = \sum_{S\in \S} w(S) \D_S$.
If it does, the decomposition is unique.
\end{thm}
\begin{proof}
Suppose $\D = \displaystyle \sum_{S\in \S} w(S) \D_S$ for some
split system $\S$ that is compatible with respect to a circular ordering $\C=\{x_1,x_2,\ldots,x_n\}$.
Choose $i < j < k < l$ and $S = A|B \in \S$. One can check that 
\begin{equation}\label{eq:kalcond}
D_S(x_i, x_k) + D_S(x_j, x_l) - D_S(x_i,x_j) - D_S(x_k,x_l) =
\begin{cases}
2 & x_i,x_j \in A, x_k,x_l \in B, \\
0 & \mathrm{otherwise}.
\end{cases}
\end{equation}
so $D$ satisfies the Kalmanson condition. 

Conversely, assume $\D$ is Kalmanson with respect to the circular ordering $\{x_1,\ldots,x_n\}$. Define
\[
2 \alpha(i,j) =\D(x_i, x_j) + \D(x_{i-1},x_{j-1}) - \D(x_i,x_{j-1}) - \D(x_{i-1},x_j).
\]
The Kalmanson condition shows this is non-negative. 

Recall that $S_{i,j} := \{x_i, x_{i+1}, \ldots, x_{j-1}\} | \{x_j, x_{j+1}, \ldots, x_{i-1}\}$.
The system $\S = \{S_{i,j}\}_{i<j}$ is clearly circular. We claim 
\begin{equation}\label{eq:claim}
\D = \displaystyle \sum_{i < j} \alpha(i,j) \D_{S_{i,j}}.
\end{equation}
To see this, rewrite the right hand side of \eqref{eq:claim}, expanding the $\alpha(i,j)$ and grouping together the coefficients of each $D(x_i,x_j)$. 
This gives $D = \sum_{i < j} \D(x_i, x_j) c_{i,j}$ with
\begin{equation}\label{eq:cij}
2 c_{i,j} = \D_{S_{i,j}} + \D_{S_{i+1,j+1}} - \D_{S_{i+1,j}} - \D_{S_{i,j+1}}.
\end{equation}
Now $c_{i,j}(x_k,x_l) = \delta_{ik}\delta_{jl}$. 
This proves the correctness of \eqref{eq:claim} and thus shows that $D$ comes from a weighted circular split system.

For a circular ordering $\C$ there are $n \choose 2$ splits in $\S$
and by \eqref{eq:cij} the dimension of metrics that are Kalmanson with respect to $\C$
is also $n \choose 2$, so for a fixed circular ordering the weighting is unique.
Now suppose $\D$ is Kalmanson with respect to two distinct circular orderings $\C_1, \C_2$.
Let $\S_i$ be the split system given by $\C_i$, and consider the decomposition $\displaystyle \D = \sum_{k<l} \alpha(k,l) \D_{S_{k,l}}$ with respect to $\C_1$.
If $S_{i,j}$ is circular with respect to $\C_1$ but not with respect
to $\C_2$, then without loss of generality there exists some $k,l$ with $i < k < j < l$ such that $\{x_i, x_k, x_j, x_l\}$ is cyclic with respect to $\C_1$
and $\{x_i, x_j, x_k, x_l\}$ is cyclic with respect to $\C_2$. This implies
\[
\D(x_i, x_j) + \D(x_k, x_l) \geq \D(x_i,x_k) + \D(x_j, x_l) \geq \D(x_i, x_j) + \D(x_k, x_l),
\]
where the first inequality comes from the Kalmanson condition on $\C_1$ and the second comes from the Kalmanson condition on $\C_2$. So we have equality, and by \eqref{eq:kalcond},
\[
0 = D(x_i,x_j) + D(x_k,x_l) - D(x_i,x_k) - D(x_k,x_l) = 2 \sum_{\stackrel{S = A|B \in \S}{ik \in A, jl \in B}} w(S) \geq w(S_{i,j}),
\]
where the inequality follows since $S_{i,j}$ is in the summand. 
So $\alpha(i,j) = w(S_{i,j}) = 0$, and the only nonzero terms in the decomposition of $\D$ with respect to $\C_1$
correspond to splits in $\S_1$ which are also splits in $\S_2$.
This shows the decomposition of $D$ is unique, and thus the map $\nu$ from weighted circular split systems to Kalmanson metrics given by 
\[
\nu(\S,w)(x,y) = \sum_{S \in \S} w(S) D_S(x,y)
\]
is a bijection.
\end{proof}

Let $\xi = \nu^{-1}$ be the map that takes a Kalmanson metric to the weighted circular split system that describes it, 
and let $D$ be Kalmanson with $\xi(D) = (\S,w)$.
We want to find a PC-tree $\T$ such that $\beta(\T) = \S$ as this would provide a nice encapsulation of the ``treeness'' of our metric, but by Proposition \ref{thm:rect2}
such a tree exists if and only if $\S$ is an unrooted family, which is not necessarily the case. There is, however, a canonical best choice.
\begin{thm}\label{thm:bestfitPC}
Let $D$ be a Kalmanson metric. There is a unique PC-tree $\T$ and weighting function $w:\beta(\T) \to \mathbf{R}_{\geq 0}$ 
such that the weighted circular split system $(\beta(\T),w)$ gives rise to $D$, 
and such that the number of zero weights $|\{S \in \beta(\T) | w(S)=0\}|$ is minimal.
\end{thm}
\begin{proof}
Define the closure map $\iota: \S \to \bar\S$, where $\bar\S$ is constructed via the following algorithm:
\begin{algorithmic}
\STATE $\bar{\S} \gets \S$
\WHILE{$\bar\S$ contains a pair of incompatible splits $A_1|B_1, A_2|B_2$}
	\STATE $\bar{\S} \gets \bar{\S} \cup \{A_1 \cap A_2 | B_1 \cup B_2, A_1 \cap B_2 | A_2 \cup B_1, A_2 \cap B_1 | A_1 \cup B_2, B_1 \cap B_2 | A_1 \cup A_2\}$
\ENDWHILE
\end{algorithmic}
Since $X$ is finite the above algorithm must terminate. By construction $\bar\S$ is a circular split system and an unrooted family, and if $\S'$ is an unrooted family with $\S' \supseteq \S$, then we must also have $\S' \supseteq \bar\S$. By Proposition \ref{thm:rect2} there is a unique PC-tree $\T$ with $\beta(\T) = \bar\S$. We have shown that if $\T'$ is another PC-tree with $\beta(\T') \supseteq \S$, then $\beta(\T') \supset \beta(\T)$, so in a well-defined sense $\T$ is the ``best-fit'' PC-tree for $D$.
Let $\xi(D) = (\S, w)$ be the weighted circular split system corresponding to $D$ and let $\bar{w}$ be a weighting on $\bar\S$ given by extending $w$ as
\[
\bar{w}(S) = \begin{cases} w(S) & S \in \S, \\ 0 & S \in \bar{\S} - \S.\end{cases}
\]
Then $\nu(\bar\S, \bar{w})=D$ and if $(\S', w')$ is a weighted circular split system with $\xi((\S', w')) = D$, then $\S' \supset \bar\S$ and $w' = w$ on $\S$, $w' = 0$ on $\S'-\S$.
\end{proof}
We now explore how this construction looks on the affine side.
\begin{defn}\label{def:gromov}
Let $\D$ be a metric on $X$ and choose $r \in X$. 
The {\em Gromov product based at $r$} is defined by 
\begin{equation}\label{eq:gromovprod}
2 \phi_r\left(\D\right)(x,y) = \D(x,y) - \D(x,r) - \D(y,r)\ \ \forall x,y\in X \setminus \{r\}. 
\end{equation}
\end{defn}
The Gromov product is also known as the {\em Farris transform} \cite{Dress2007,Farris1972} in phylogenetics.
It is easy to check that the map
\[
\psi_r(R)(x,y) = 2 R(x,y) - R(x,x) - R(y, y).
\]
satisfies $\psi_r \circ \phi_r(D) = D$ and so is its inverse.
\begin{defn}\label{def:robinsonian}
A matrix $\R$ is \emph{Robinsonian} over $X$ if 
there exists a linear ordering $\prec$ of $X$ such that
\[
max\{\R(x,y),\R(y,z)\} \leq \R(x,z) \ \ \ \forall x \preceq y \preceq z.
\]
$\R$ is a \emph{strong Robinsonian matrix} if, in addition, for all $w \preceq x \preceq y \preceq z$,
\begin{align}
\label{eq:strob1} R(x,y) = R(w,y) \implies R(x,z) = R(w,z), \\
\label{eq:strob2} R(x,y) = R(x,z) \implies R(w,y) = R(w,z).
\end{align}
\end{defn}
In \cite{Christopher1996} it is shown that if $D$ is Kalmanson then $\phi_r(D)$ is a strong Robinsonian matrix. Here, we give a slightly more precise characterization of the image.
\begin{lemma}\label{lem:bonus}
Let $D$ be a Kalmanson dissimilarity map and $R = \phi_r(D)$. Then $R$ is a strong Robinsonian matrix with the following properties:
\begin{enumerate}
\item $R(x,y) \leq 0$ for all $x,y \in X$,
\item For every $w \preceq x \preceq y \preceq z$, 
\begin{equation}\label{eq:RtoK}
R(x,y) + R(w,z) \leq R(x,z) + R(w,y).
\end{equation}
Furthermore, $\phi_r$ is a bijection from Kalmanson dissimilarities to the space of these matrices.
\end{enumerate}
\end{lemma}
\begin{proof}
Suppose $D$ is Kalmanson with respect to the order $\{x_1,x_2,\ldots,x_n,r\}$ and $R = \phi_r(D)$. It is immediate from the definition of the Gromov product \eqref{eq:gromovprod} and the Kalmanson condition \eqref{eq:kal} that $\phi_r(D)$ satisfies the above conditions with linear ordering $x_1 \prec x_2 \prec \ldots \prec x_n$. For $w \preceq x \preceq y \preceq z$, 
\[
2(R(x,z) - R(x,y)) = D(x,z) + D(y,r) - D(x,y) - D(z,r) \geq 0,
\]
so $R(x,z) \geq R(x,y)$. Similarly $R(x,z) \geq R(y,z)$, so $R$ is Robinsonian. Now assume $R(x,y) = R(x,z)$. Then \eqref{eq:RtoK} gives $R(w,z) \leq R(w,y)$, and since $R$ is Robinsonian we also have the reverse inequality, so $R(w,z) = R(w,y)$. Similarly if $R(x,y) = R(w,y)$ then $R(x,z) = R(w,z)$, so $R$ is strong.

Conversely, let $R$ be a strong Robinsonian matrix satisfying \eqref{eq:RtoK}. Then $\psi_r(R)$ clearly satisfies the Kalmanson conditions. Also,
\[
\psi_r(x,y) = (R(x, y)-R(x,x)) + (R(x,y) - R(y,y)) \geq 0,
\]
and $\psi_r(x,x)=0$ for all $x,y \in X$. So $\psi_r(R)$ is a Kalamanson dissimilarity. The maps $\phi_r$ and $\psi_r$ are inverses, completing the proof.
\end{proof}
Therefore the image of $\phi_r$ consists of negative strong Robinsonian matrices satisfying a kind of four-point condition \eqref{eq:RtoK}.

Next we define the affine analogue of weighted circular split systems.
\begin{defn}\label{def:iph}
Let $\P$ be a pyramid. A function $f: \P \to \mathbf{R}$ is an \emph{indexing function} if $A \subset B \implies f(A) < f(B)$ for all $A, B \in \P$. We call $(\P,f)$ an \emph{indexed pyramid}.
\end{defn}
\begin{defn}\label{def:ml}
A subset $A \subseteq X$ is \emph{maximally linked} \cite{Bertrand2002}  with respect to a Robinsonian matrix $R$ if there exists $d \in \mathbf{R}$ such that $\R(x,y) \leq d$ for all $x,y \in A$,
and $A$ is maximal in this way. If $A$ is such a set, define the \emph{diameter} of $A$ to be $diam(A) = \max_{x,y \in A} R(x,y)$. 
\end{defn}
Let $\mathcal{M}(\R)$ denote the set of maximally linked sets with respect to Robinsonian matrix $\R$ and define the function $f:\MM(\R) \to \mathbf{R}$ by $f(A) = diam(A)$. 
\begin{prop}\label{prop:RtoIP}
The map $\tau: \R \to (\MM(\R),f)$ is a bijection from Robinsonian matrices to indexed prepyramids, and from strong Robinsonian matrices to indexed pyramids.
\end{prop}
\begin{proof}
Suppose $A \in \MM(\R)$ for $\R$ Robinsonian and let $a,b \in A$ be the leftmost and rightmost points in $A$. Then for every $a \preceq x \prec y \preceq b$, $\R(x,y) \leq R(a,b)$, 
so $diam(A) = M(a,b)$. This shows $x \in A$ for all $a \preceq x \preceq b$, so every set in $\mathcal{M}(R)$ is an interval.
Now suppose $A,B \in \MM(\R)$ with $A \subset B$. Let $A=[x_1,y_1], B=[x_2,y_2]$. Then $x_2 \preceq x_1 \preceq y_1 \preceq y_2$, and 
\[
f(A) = diam(A) = \R(x_1,y_1) < \R(x_2, y_2) = diam(B) = f(B),
\]
where the inequality follows from the fact that $A$ is a maximally-linked set. So $f$ is an index and $\tau(\R)$ is an indexed prepyramid.

Conversely, consider the map $\mu$ from indexed prepyramids to matrices given by 
\[
\mu(\P,f)(x,y) = \min_{\stackrel{A \in \P}{x,y \in A}} f(A).
\]
Let $\R=\mu(\P,f)$. Given $x \preceq y \preceq z$, let $E = \{A \in \P | x,y \in A\}$, $F = \{A \in \P | x,z \in A\}$. Then $F \subseteq E$, so
\[
\R(x,y) = \min_{A \in E} f(A) \leq \min_{A \in F} f(A) = \R(x,z).
\]
Similarly $\R(y,z) \leq \R(x,z)$, so $\R$ is Robinsonian. It is easy to check that $\P$ consists precisely of the sets that are maximally linked with respect to $R$, so $\tau$ and $\mu$ are inverses.

Now suppose $\R$ is a strong Robinsonian matrix. We must show $\tau(\R)$ is closed under intersection. 
Let $A=[a_1,b_1],B=[a_2,b_2]$ be sets in $\P$, suppose $a_1 \prec a_2 \preceq b_1 \prec b_2$ and let $C = A \cap B = [a_2, b_1]$.
We will show $C$ is a maximally linked set with diameter $\R(a_2,b_1)$. 
If $x \succ b_1$, the Robinsonian condition gives $\R(a_2,x) \geq R(a_2,b_1)$. If there was equality then by the strong Robinsonian condition $\R(a_1,x)=R(a_1,b_1)$
and $x \in A$, a contradiction. Similarly, there is no $x \prec a_2$ with $R(x,b_1) = R(y,a_1)$, so $C \in \P$ and $\P$ is closed under intersection.

Conversely, suppose $(\P,f)$ is an indexed pyramid and let $R = \mu((\P,f))$. Because $\P$ is closed under intersection, for each $A \subseteq X$ there is a unique $\bar{A} \in \P$ such that $A \subseteq \bar{A}$, and $\bar{A} \subseteq B$ for all $A \subseteq B \in \P$. This follows immediately from taking $\bar{A} = \bigcap_{A \subseteq B \in \P}B$. So now, suppose $w \prec x \prec y \prec z$ and $R(x,y) = R(x,z)$. The set $A := \overline{\{x,y\}} \cap \overline{\{w,y\}}$ is in $\P$ since $\P$ is closed under intersection. $f(\overline{\{x,y\}}) = f(\overline{\{x,z\}})$ which implies $z \in \overline{\{x,y\}}$. Now $x,y \in A$ implies $\overline{\{x,y\}} \in A$, so $z \in A$. But then $z \in \overline{\{w,y\}}$ which gives $R(w,z) = R(w,y)$. A similar argument shows $R(x,y) = R(w,y) \implies R(x,z) = R(w,z)$, so $R$ is a strong Robinsonian matrix.
\end{proof}
Given two elements $A,B \in \P$ we say $B$ is a \emph{predecessor} of $A$ if $A \subset B$ and there does not exist $C \in \P$ such that $A \subsetneq C \subsetneq B$. 
\begin{lemma}\label{lem:10}
Let $\P$ be a pyramid. Then each set in $\P$ has at most two predecessors.
\end{lemma}
\begin{proof}
Suppose there is an $A = [a,b] \in \P$ with three distinct predecessors $B_i = [a_i, b_i]$, $i=1,2,3$. 
Because $\P$ is closed under intersection $B_i \cap B_j = A$ so either $a_i = a$ or $b_i=b$.
By the pigeonhole principle two of the $B_i$s must share an endpoint, so assume $a_1 = a_2 = a$.
Then either $B_1 \subset B_2$ or $B_2 \subset B_1$ contradicting the fact that each $B_i$ is a predecessor of $A$.
\end{proof}
For a set $A \in \P$, let $P_A = \{P_i\}$ denote the predecessors of $A$. If $(\P,f)$ is an indexed pyramid, we define the map $w: \P \to \mathbf{R}$ as the unique function satisfying
\begin{equation}\label{eq:w}
w(A) = \begin{cases} -f(A) & \mathrm{if\ }|P_A|=0,\\
  -f(A) + f(P_1) & \mathrm{if\ }|P_A| = 1, \\
  -f(A) + f(P_1) + f(P_2) - f(\overline{P_1 \cup P_2}) & \mathrm{if\ }|P_A|=2.
  \end{cases}
\end{equation}
By Lemma \ref{lem:10} this is well-defined. 
\begin{prop}
Let $R$ be a negative strong Robinsonian matrix satisfying the Robinsonian four-point condition, and take $\tau(R) = (\P,f)$. Then $f$ is negative, and $w(A) \geq 0$ for all $A \in \P$. Furthermore, every such indexed pyramid lies in the image of $\tau$.
\end{prop}
\begin{proof}
Let $R$ be a negative Robinsonian matrix satisfying \eqref{eq:RtoK}. Clearly \eqref{eq:w} holds for $|P_A|=0$ because $f$ is negative, and holds for $|P_A|=1$ because $P_1 \supset A \implies f(P_1) > f(A)$. So now assume $A = [x,y]$ has two predecessors. By the argument in Lemma \ref{lem:10} these must be of the form $B_1 = [w,y]$ and $B_w = [x,z]$ for some $w \prec x \prec y \prec z$, so
\begin{align*}
w(A) &= -f([x,y]) + f([w,y]) + f([w,z]) - f(\overline{\{w,z\}}) \\
 &= -R(x,y) + R(w,y) + R(w,z) - R(w,z) \\
&\geq 0,
\end{align*}
because $R$ satisfies the Robinsonian four-point condition. This argument is reversible, so we see $\tau$ really is a bijection.
\end{proof}
The requirement that $w(A) \geq 0$ for $A$ with two predecessors \eqref{eq:w} is thus a kind of four-point property for pyramids, and we will refer to it as such later.

Let $\eta_r$ be the map sending $(\P,f)$ to the weighted circular split system $(\S,w')$ given by $\S = \{\delta_r(A) | A \in \P\}$, $w'(\delta_r(A)) =w(A)$.
\begin{prop}
If $D$ is a Kalmanson metric, then $\nu \circ \eta_r \circ \tau \circ \phi_r$ is the identity map.
\end{prop}
\begin{proof}
Let $D' = \nu \circ \eta_r \circ \tau \circ \phi_r(D)$  and for $A \in \P$, let $O_A = \{B \in \P | A \subseteq B\}$ be the sets over $A$. A split $\delta_r(A)$ separates $x, y \in X \setminus \{r\}$
if $x \in A$ or $y \in A$ but not both. So the split pseudometric $D_{\delta_r(A)}(x,y)$ is $1$ iff $A \in O_{\{x\}} \setminus O_{\{x,y\}}$
or $A \in O_{\{y\}} \setminus O_{\{x,y\}}$. Then
\begin{equation}\label{eq:theend}
D'(x,y) = \displaystyle \sum_{A \in O_{\{x\}}}
w(A) - \displaystyle \sum_{A \in O_{\{x,y\}}} w(A)
+ \displaystyle \sum_{A \in O_{\{y\}}} w(A) -
\displaystyle \sum_{A \in O_{\{ x,y \} }}w(A).
\end{equation}
Now $O_A = O_{\bar{A}}$, so by an easy induction 
\[
\sum_{B \in O_A} w(B) = \sum_{B \in O_{\bar{A}}} w(B) = -f(\bar{A}),
\]
and $D'(x,y) = 2 f\left( \overline{\{x,y\}} \right) - f(\overline{\{x\}}) - f(\overline{\{y\}})$.
Since $f(\bar{A}) = diam(A)$, by the definition of the Gromov transform,
\[
D'(x,y) = \D(x,y) - \D(x,r) - \D(y,r) + \D(x,r) + \D(y,r) = \D(x,y).
\]
To compute $\D'(x,r)$ we note $x$ and $r$ are separated
by a split $\delta_r(A)$ iff $x \in A$, so
\[
D'(x,r) = \sum_{A \in O_{\overline{\{x\}}}} w(A) = -f(\overline{\{x\}}) = -\phi_r(D)(x,x) = D(x,r).
\] 
\end{proof}

Let $\R$ be a Robinsonian matrix over $X$. If the prepyramid $\P$ in $\tau(\R)$ is a rooted set family, Proposition \ref{thm:PQconstruct} shows there exists a PQ-tree $\T$ such that $\alpha(\T)=\P$. 
Unfortunately this is usually not the case, so
we seek instead to find a ``best fit'' tree. Analogous to the projective case, 
we construct the \emph{rooted closure} $\bar\P$ of $\P$ with the following algorithm.
\begin{algorithmic}
\STATE $\bar{\P} \gets \P$
\WHILE{$\bar\P$ contains a pair of incompatible sets $A,B$}
	\STATE $\bar{\P} \gets \bar{H} \cup \{A \cup B, A \cap B, A \setminus B, B \setminus A\}$
\ENDWHILE
\end{algorithmic}
Let $\theta$ be the closure map sending $\P$ to $\bar\P$. 
We then have the affine analog of Theorem \ref{thm:bestfitPC}:
\begin{lemma}\label{prop:bestfitPQ}
The PQ-tree $\T$ with $\alpha(\T) = \theta(\P)$ is the unique tree with $\alpha(\T) \supseteq \P$ that minimizes $|\alpha(\T)|$.
\end{lemma}
It remains to show that $\theta$ commutes with the rest of the diagram. Let $D$ be a Kalmanson metric, $\S$ the corresponding 
split system and $\P$ the associated indexed pyramid. 
There is not necessarily a bijection between the intervals in $\P$ and the splits in $\S$; this can be seen, for example, because $\P$ 
is closed under intersection while $\S$ need not be. The splits in $\P$ that are not in $\S$ will get assigned weight zero
by the map $\eta_r$, which is why the lower rectangle commutes, but the maps $\theta$ and $\iota$ forget about the weights 
so it's not clear that $\theta \circ \delta_r = \iota \circ \psi_r$. Fortunately, for pyramids that arise from Kalmanson metrics
this bijection holds.
\begin{lemma}
$\delta_r \circ \theta \circ \tau \circ \phi_r(D) = \iota \circ \xi(D)$ for all Kalmanson metrics $D$.
\end{lemma}
\begin{proof}
Let $(\P,f) = \tau \circ \phi_r(D)$ and let $(\S,w)$ the corresponding weighted split system.
Suppose $A \in \P$ but $\delta_r(A) \notin \S$, or equivalently $f(A)=0$. $A = [a,b]$ is an interval with respect to the Robinsonian metric.
If $c \succ b$ then $M(a,c) > M(a,b)$ because $c \notin A$, so
\[
0 < M(a,c) - M(a,b) = f(a,c) - f(a,b) = -\sum_{\stackrel{B \in \P}{a,c \in B}} w(B) + \sum_{\stackrel{B \in \P}{a,b \in B}} w(B).
\]
The first summand is a subset of the second, so there exists $B \in \P$ with $a,b \in B$, $c \notin B$ and $w(B) > 0$. Letting $c$ be the smallest element with $c \succ b$ shows there exists 
a set $B = [x,b] \in \P$ with $w(B) > 0$ and $x \prec a$. Similarly there exists a set $C = [a,y] \in \P$ with $w(C) > 0$ and $y \succ b$.
So $A = B \cap C$ for sets $B,C \in \P$ that correspond to splits of positive weight in $\S$,
and thus $\delta_r(A) \in \iota \circ \eta(\S)$. This completes the proof.
\end{proof}
We are now ready to state our final result that summarizes the
bijections described above. 
Let $PC$ be the set of all PC-trees, $CUF$ the set of all
circular, unrooted split families, $WCSS$ the set of all weighted
circular split systems, and $K$ the set of all Kalmanson metrics,
all over $X$. Let $PQ$ be the set of all PQ-trees,
$PRF$ the set of pyramidal rooted families, $IP$ the set of
negative indexed pyramids satisfying the pyramidal four-point condition, 
and $SR$ the set of negative strong Robinsonian matrices satisfying the Robinsonian four-point condition, all over $X \setminus \{r\}$.
\begin{thm} \label{thm:rectangle}
The following diagram commutes:
\[
\xymatrix{
PQ \ar@<.5ex>[d]^\alpha  \ar@<.5ex>[r]^{\kappa_r}  & PC \ar@<.5ex>[d]^\beta \ar@<.5ex>[l]^{\lambda_r} \\
PRF   \ar@<.5ex>[r]^{\delta_r} \ar@<.5ex>[u] & CUF \ar@<.5ex>[l]^{\gamma_r} \ar@<.5ex>[u] \\
IP \ar@<0ex>[u]^\theta \ar@<.5ex>[r]^\eta \ar@<.5ex>[d]^\mu
& WCSS \ar@<0ex>[u]_\iota \ar@<.5ex>[l] \ar@<.5ex>[d]^\nu \\
SR \ar@<.5ex>[u]^\tau \ar@<.5ex>[r]^{\psi_r} 
& K \ar@<.5ex>[u]^\xi \ar@<.5ex>[l]^{\phi_r}
}\]
\end{thm}
This gives a way of constructing the best-fit PC-tree for a given Kalmanson metric $D$: 
\[
\T = \kappa_r \circ \alpha^{-1} \circ \theta \circ \tau \circ \phi_r(D).
\]

An example illustrating  Theorem \ref{thm:rectangle} is shown in
Figure 6. The PC-tree in the upper right reveals the tree
structure in the Kalmanson metric (see also Figure 2).

\begin{figure}
\label{fig6}
\begin{center}
\includegraphics[scale=0.35]{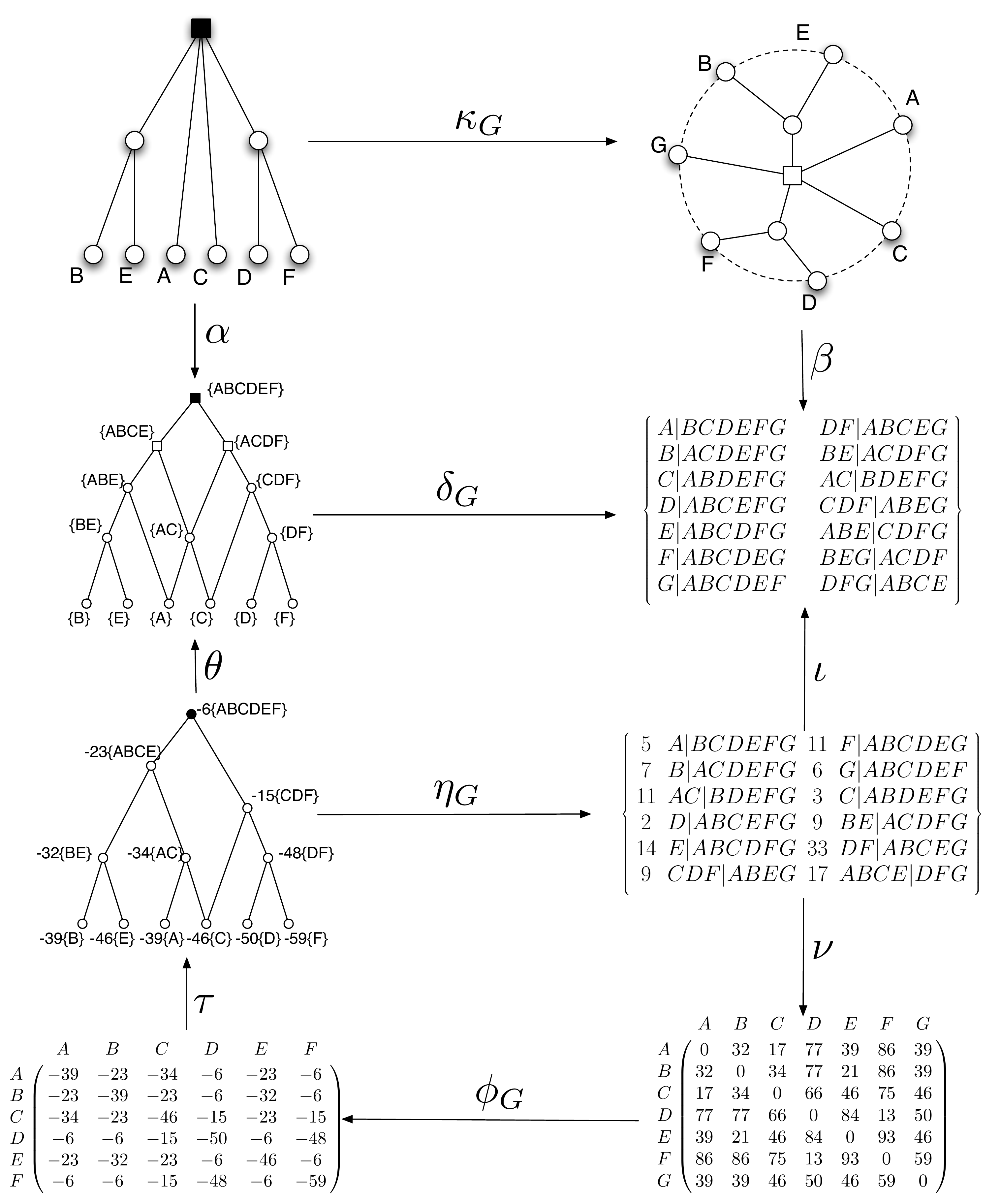}
\end{center}
\caption{An example illustrating Theorem \ref{thm:rectangle}.}
\end{figure}
As a final remark, we note that the geometry of Kalmanson metrics is also explored in
\cite{Terhorst2011} where the Kalmanson complex is described. PC-trees are faces in this complex, and it should be interesting to understand their combinatorics in the face lattice.

\bibliographystyle{amsplain}
\newpage
\bibliography{zproof}

\providecommand{\bysame}{\leavevmode\hbox to3em{\hrulefill}\thinspace}
\providecommand{\MR}{\relax\ifhmode\unskip\space\fi MR }
\providecommand{\MRhref}[2]{%
  \href{http://www.ams.org/mathscinet-getitem?mr=#1}{#2}
}
\providecommand{\href}[2]{#2}
\begin{thebibliography}{10}

\bibitem{Bertrand2002}
P~Bertrand and MF~Janowitz, \emph{Pyramids and weak hierarchies in the ordinal
  model for clustering}, Discrete Applied Mathematics \textbf{122} (2002),
  55--81.

\bibitem{Booth1976}
K~Booth and G~Lueker, \emph{Testing for the consecutive ones property, interval
  graphs, and graph planarity using {PQ}-tree algorithms}, Journal of Computer
  and System Sciences \textbf{13} (1976), no.~3, 335--379.

\bibitem{Bryant2004}
D~Bryant and V~Moulton, \emph{{NeighborNet}: An agglomerative method for the
  construction of planar phylogenetic networks}, Molecular Biology And
  Evolution \textbf{21} (2004), 255--265.

\bibitem{Bryant2007}
D~Bryant, V~Moulton, and A~Spillner, \emph{Consistency of the neighbor-net
  algorithm}, Algorithms for Molecular Biology \textbf{2} (2007), 8.

\bibitem{Buneman1971}
P~Buneman, \emph{The recovery of trees from measures of dissimilarity},
  Mathematics in the Archaeological and Historical Sciences (FR~Hodson,
  DG~Kendall, and P~Tautu, eds.), Edinburgh University Press, 1971,
  pp.~387--395.

\bibitem{Chepoi1998}
V~Chepoi and B~Fichet, \emph{A note on circular decomposable metrics},
  Geometrica Dedicata \textbf{69} (1998), 237--240.

\bibitem{Christopher1996}
G~Christopher, M~Farach, and M~Trick, \emph{The structure of circular
  decomposable metrics}, Lecture Notes in Computer Science, vol. 1136,
  Springer, New York, 1996, pp.~406--418.

\bibitem{Dewey2006}
C~Dewey and L~Pachter, \emph{Evolution at the nucleotide level: the problem of
  multiple whole genome alignment}, Human Molecular Genetics \textbf{15}
  (2006), R51--R56.

\bibitem{Diday1986}
E~Diday, \emph{Multidimensional data analysis}, ch.~Orders and overlapping
  clusters by pyramids, pp.~201--234, DWO Press, Leiden, 1986.

\bibitem{Dress1997}
A~Dress, \emph{Towards a theory of holistic clustering}, Mathematical
  Hierarchies and Biology, DIMACS, 1997.

\bibitem{Dress2007}
A~Dress, KT~Huber, and V~Moulton, \emph{Some uses of the {Farris} transform in
  mathematics and phylogenetics-- a review}, Annals of Combinatorics
  \textbf{11} (2007), 1--37.

\bibitem{Edmonds1977}
J~Edmonds and R~Giles, \emph{A min-max relation for submodular functions on
  graphs}, Studies in Integer Programming (Hammer, Johnson, Korte, and
  Nemhauser, eds.), North-Holland, 1977, pp.~185--204.

\bibitem{Eslahchi2010}
C~Eslahchi, M~Habibi, R~Hassanzadeh, and E~Mottaghi, \emph{{MC-Net: a method
  for the construction of phylogenetic networks based on the Monte-Carlo
  method}}, BMC Evolutionary Biology \textbf{10} (2010), 254.

\bibitem{Fakcharoenphol2003}
J~Fakcharoenphol, S~Rao, and K~Talwar, \emph{A tight bound on approximating
  arbitrary metrics by tree metrics}, Proceedings of the thirty-fifth annual
  ACM symposium on the Theory of Computing, 2003, pp.~448--455.

\bibitem{Farris1972}
JS~Farris, \emph{Estimating phylogenetic trees from distance matrices},
  American Naturalist \textbf{106} (1972), 645--668.

\bibitem{Gusfield1991}
D~Gusfield, \emph{Efficient algorithms for inferring evolutionary history},
  Networks \textbf{21} (1991), 19--28.

\bibitem{Hsu2001}
W-L Hsu, \emph{{PC-trees vs. PQ-trees}}, Lecture Notes in Computer Science
  (J~Wang, ed.), vol. 2108, 2001, pp.~207--217.

\bibitem{McConnell2003}
W-L Hsu and RM~McConnell, \emph{{PC} trees and circular-ones arrangements},
  Theoretical Computer Science \textbf{296} (2003), 99--116.

\bibitem{Huson2005b}
D~Huson and D~Bryant, \emph{Application of phylogenetic networks in
  evolutionary studies}, Molecular Biology and Evolution \textbf{23} (2005),
  254--267.

\bibitem{Jardine1967}
CJ~Jardine, N~Jardine, and R~Sibson, \emph{The structure and construction of
  taxonomic hierarchies}, Mathematical Bioscience \textbf{1} (1967), 173--179.

\bibitem{Kalmanson1975}
K~Kalmanson, \emph{Edgeconvex circuits and the traveling salesman problem},
  Canadian Journal of Mathematics \textbf{27} (1974), 1000--1010.

\bibitem{Levy2010}
D~Levy and L~Pachter, \emph{The neighbor-net algorithm}, Advances in Applied
  Mathematics \textbf{47} (2011), 240--258.

\bibitem{ASCB2005}
L~Pachter and B~Sturmfels (eds.), \emph{Algebraic statistics for computational
  biology}, Cambridge University Press, 2005.

\bibitem{Robinson1951}
WS~Robinson, \emph{A method for chronologically ordering archaeological
  deposits}, American Antiquity \textbf{16} (1951), 293--301.

\bibitem{Semple2003}
C~Semple and M~Steel, \emph{Phylogenetics}, {\rm Oxford Lecture Series in
  Mathematics and its Applications}, vol.~24, Oxford University Press, Oxford,
  2003.

\bibitem{Semple2004}
\bysame, \emph{Cyclic permutations and evolutionary trees}, Advances in Applied
  Mathematics \textbf{32} (2004), 669--680.

\bibitem{Shih1999}
Wei-Kuan Shih and Wen-Lian Hsu, \emph{A new planarity test}, Theoretical
  Computer Science \textbf{223} (1999), 179 -- 191.

\bibitem{Terhorst2011}
J~Terhorst, \emph{{The Kalmanson complex}}, arXiv.org:abs/1102.3177 (2011).

\bibitem{White2007}
WT~White, SF~Hills, R~Gaddam, BR~Holland, and D~Penny, \emph{Treeness
  triangles: visualizing the loss of phylogenetic signal}, Molecular Biology
  and Evolution \textbf{24} (2007), 2029--2039.

\end{thebibliography}
\newpage
\begin{table}
\caption{Nomenclature and abbreviations.}
\begin{tabular}{l l}
AT & Affine (rooted) $X$-trees with root $r$ \\
PT & Projective (unrooted) $X$-trees\\
H & Hierarchies over $X \setminus \{r\}$\\
PSS & Pairwise compatible split systems over $X$\\
U & Ultrametrics\\
TM & Tree metrics\\
PQ & PQ-trees over $X \setminus \{r\}$ \\
PC & PC-trees over $X$ \\
PRF & Pyramids that are rooted families over $X \setminus \{r\}$ \\
CUF & Circular split systems that are unrooted families over $X$ \\
IP & Negative indexed pyramids satisfying \\
 & the pyramidal four-point condition over $X \setminus \{r\}$ \\
WCSS & Weighted circular compatible split systems over $X$\\
SR & Negative strong Robinsonian matrices satisfying \\
 & the Robinsonian four-point condition over $X \setminus \{r\}$\\
K & Kalmanson metrics over $X$
\end{tabular}
\end{table}
\end{document}